\documentclass[12pt,leqno]{amsart}
\usepackage{amsfonts,amsthm,amsmath,comment,xcolor}
\usepackage{url}
\usepackage{tikz}
\usepackage{tkz-berge}
\usetikzlibrary{cd, arrows.meta}

\theoremstyle{plain}
\newtheorem{thm}{Theorem}[section]
\newtheorem{prop}[thm]{Proposition}

\newtheorem{lemma}[thm]{Lemma}
\newtheorem{cor}[thm]{Corollary}

\theoremstyle{definition}

\newtheorem{defin}[thm]{Definition}
\newtheorem{example}[thm]{Example}
\newtheorem{prob}[thm]{Problem}
\newcommand{\CC}{\mathbb{C}}

\newcommand{\note}{\vspace{2ex}\noindent{\bf Note.\quad}}

\newcommand{\bo}{\mbox{\boldmath $0$}}
\newcommand{\bu}{\mbox{\boldmath $u$}}

\newcommand{\bv}{\mbox{\boldmath $v$}}
\newcommand{\bw}{\mbox{\boldmath $w$}}

\newcommand{\bC}{\mathbb{C}}

\newcommand{\cL}{\mathcal{L}}

\newcommand{\cT}{\mathcal{T}}

\newcommand{\nmat}{\mbox{{\rm Mat}}_n(\bC)}
\newcommand{\xmat}{\mbox{{\rm Mat}}_X(\bC)}

\newcommand{\mat}{\mbox{{\rm Mat}}}

\newcommand{\spn}{\mbox{{\rm Span}}}

\newcommand{\sym}{\mbox{{\rm Sym}}}

\newfont{\bg}{cmr12 scaled\magstep4}
 
 \newcommand{\bigzerou}{\smash{\lower2ex\hbox{\bg 0}}}

\begin{document}

%\title{The Terwilliger algebra of a Hamming digraph}
\title[The Terwilliger algebra of digraphs I]{The Terwilliger algebra of digraphs I ---Hamming digraph $H^*(d,3)$---}

\author[Miezaki]{Tsuyoshi Miezaki}
\address
{
	Faculty of Science and Engineering, 
	Waseda University, 
	Tokyo 169-8555, Japan\\ %(Corresponding author)
}
\email{miezaki@waseda.jp}

\author[Suzuki]{Hiroshi Suzuki}
\address
{
	College of Liberal Arts, 
	International Christian University, 
	Tokyo 181-8585, Japan\\ %(Corresponding author)
}
\email{hsuzuki@icu.ac.jp}

\author[Uchida]{Keisuke Uchida}
\address
{
	School of Fundamental Science and Engineering, 
	Waseda University, 
	Tokyo 169-8555, Japan\\ %(Corresponding author)
}
\email{u\_kei\_orange@akane.waseda.jp}

\date{}
\maketitle

\begin{abstract}
In the present paper, we define the Terwilliger algebra of digraphs.
Then, we determine the irreducible modules of the Terwilliger algebra of a Hamming digraph $H^*(d,3)$. As is well known, the representation of the Terwilliger algebra of a binary Hamming graph $H(d,2)$ is closely related to that of the Lie algebra $\mathit{sl}_2(\bC)$. We show that in the case of $H^*(d,3)$, it is related to that of the Lie algebra $\mathit{sl}_3(\bC)$.
We also identify the Terwilliger algebra of $H^*(d,3)$ as the $d$ symmetric tensor algebra of $\mat_3(\bC)$.
\end{abstract}

{\small
\vspace{-0.15in}

\quad
}

%\noindent
%2010 {\it Mathematics Subject Classification}.
%Primary 11F30;
%Secondary 20D08, 11F27.\\ \quad

\section{Introduction}

\begin{defin}
Let $N = \{a_1, a_2, \ldots, a_n\}$. The Hamming graph $H(d,n)$ or more specifically, the Hamming graph of diameter $d$ on the set $N$, $H(d, N)$ is defined by $X$, the set of vertices, and $\tilde{E}$, the set of edges.
\begin{align*}
X & = \{(x_1, \ldots, x_d)\mid \text{ for all $i$, $x_i\in N$}\},\\
\tilde{E} & = \{\{x,y\}\mid \text{exactly 1 coordinate $i$, $x_i \neq y_i$}\} \subseteq X\times X.
\end{align*}
\end{defin}

We use a trivial direction on each coordinate, $0\to 1\to 2 \to 0$ if $N = \mathbb{F}_3 = GF(3)$.

\begin{defin}
The directed graph $H^*(d,3)$ is defined by $X$, the set of vertices, and $E$, the set of arcs.
\begin{align*}
X & = \{(x_1, \ldots, x_d)\mid \text{ for all $i$, $x_i\in \mathbb{F}_3$}\},\\
E & = \{(x,y)\mid \text{exactly 1 coordinate $i$, $x_i + 1 = y_i$}\} \subseteq X\times X.
\end{align*}

Let $\Gamma = (X,E)$ be $H^*(d,3)$ and $\xmat$ the matrix algebra whose rows and columns are indexed by the elements of $X$.

The adjacency matrix of $H^*(d,3)$, $A\in \xmat$ is defined by the following.
$$(A)_{ij} = \begin{cases} 1 & \text{if there exists an arc from $i$ to $j$},\\
0 & \text{otherwise}.
\end{cases}$$
For $x$, $y$ in $X$, $\partial(x,y)$ denotes the distance between $x$ and $y$, i.e., the smallest number of arcs connecting from $x$ to $y$. We also define the two-way distance $\tilde{\partial}(x,y) = (\partial(x,y),\partial(y,x))$, and $\Delta$, the set of all two-way distances;
$$\Delta = \{\tilde{\partial}(x,y)\mid x, y\in X\}.$$
\end{defin}

We fix a base vertex $x = (0,0, \ldots, 0)\in X$.
Let $V = \mathbb{C}^{|X|}$ be the vector space over the complex number field $\bC$ whose coordinates are indexed by the elements of $X$.  For $y\in X$, $\hat{y}$ denotes the unit vector in $V$. Let
$$X_{i,j} = \{y\mid \partial(x,y)=i, \partial(y,x)=j\}.$$
If there are $s$ ones and $t$ twos, and therefore $r = d-s-t$ zeros, then $\partial(x, y) = s+2t, \partial(y,x) = 2s+t$.
If $\tilde{\partial}(x,y) = (i,j)$, then $s = (2j-i)/3$, and $t= (2i-j)/3$.

For $(i, j)\in \Delta$, $E^*_{i,j}\in \xmat$ denotes a diagonal matrix such that
$$E^*_{i,j}(z,z) = 1, \text{ if } \tilde{\partial}(x,z) = (i,j),$$
and the zero matrix of the same size if $(i,j)\not\in \Delta$.
Then,
$$E^*_{i,j}\mathbf{1} = \{\sum \hat{y} \mid y\in X, \partial(x,y) = i, \partial(y,x) = j\}.$$

We also write
$$X_{[s,t]} = \{y\mid \text{there are $s$ ones and $t$ twos}\} = X_{s+2t,2s+t}.$$
If a vector $\bv\in V$ is a linear combination of the unit vectors corresponding to the elements in $X_{[s,t]}$, i.e., if $E^*_{[s,t]}\bv = \bv$ with $r = d-s-t$, we write
$$r(\bv) = r, \; s(\bv) = s, \text{ and }t(\bv) = t.$$
Hence, if $E^*_{[s,t]}\bv = \bv$, we write $\text{type}(\bv) = (r(\bv), s(\bv), t(\bv))$, and call it the type of $\bv$.
We set
$$E^*_{i,j} = E^*_{[(2j-i)/3, (2i-j)/3]}, \text{ and } E^*_{[s,t]} = E^*_{s+2t, 2s+t}.$$ % suggested by KU

Let $X_{i,j} = X_{[s,t]}$, where  $t = (2i-j)/3$, and $s = (2j-i)/3$.
Then,
$$|X_{[s,t]}| = \binom{d}{s+t}\binom{s+t}{s} = \binom{d}{s}\binom{d-s}{t}.$$

\begin{defin}
The {\it Terwilliger algebra} $\cT(x)$ of $H^*(d,3)$ with respect to a base vertex $x$  is a $\bC$-algebra generated by $A$, $A^\top$ and $E^*_{i,j}$ for $(i,j)\in \Delta$, and the vector space $V = \mathbb{C}^{|X|}$ is called the {\it standard module} of $\cT(x)$. We also write $\cT(H^*(d,3))$ to specify the digraph.
\end{defin}

\begin{thm}
Let $\cT(x)$ be the Terwilliger algebra of $H^*(d,3)$ with respect to a base vertex $x$. Then,
$$\cT(x) \simeq \mathrm{Sym}^{(d)}(\mat_3(\bC)).$$
Moreover, 
$$\cT(x) \simeq \mathrm{Sym}^{(d)}(\mat_3(\bC))\simeq \bigoplus_{n\in \Lambda}\mat_n(\bC), $$
where 
\begin{align*}
\Lambda = &\left\{\frac12(d-3\ell-2m+1)(m+1)(d-3\ell-m+2)\:\right|\:\\ 
&\hspace{110pt}\left. 0\leq \ell\leq \left[\frac{d}{3}\right], 0\leq m\leq \left[\frac{d-3\ell}{2}\right]\right\}.
\end{align*}
\end{thm}

We also describe the irreducible $\cT(x)$ modules in the standard module. See sections 3 and 4.

This paper is organized as follows.
In Section 2, we define several matrices derived from 
$H^*(d,3)$ and establish their connection with the Lie algebra $sl_3(\CC)$. 
In Sections 3 and 4, 
we present the irreducible representations of 
the Terwilliger algebra of $\mathcal{T}(sl_3(\CC))$. 
Finally, in Section 5, we provide a proof of Theorem 1.4. 
\section{Matrices}

\subsection{Lie algebra $sl_3(\CC)$}

We define several matrices in $\cT(x)$, which play an important role in this paper. We will see that these matrices are generators of the Lie algebra $\mathit{sl}_3(\bC)$.% on the right. We will explain the correspondence later.

In the following, for matrices $M_1$ and $M_2$, $[M_1, M_2] = M_1M_2-M_2M_1$, the Lie bracket product of $M_1$ and $M_2$.

Let $r = d-s-t$ and $H_1$, $H_2$, $H_3$, $R_1$, $R_2$, $R_3$, $L_1$, $L_2$ and $L_3$ be as follows.

\begin{align*}
H_1 & = \sum_{s,t}(r-s)E^*_{[s,t]},\;
% \sim h_1 = \begin{bmatrix}
%1 &  0 & 0 \\
%0 & -1 & 0\\
%0 & 0 & 0
%\end{bmatrix}
%, \text{where $r = d-s-t$},\\
H_2 = \sum_{s,t}(s-t)E^*_{[s,t]},\;
%\sim h_2 = \begin{bmatrix}
%0 & 0  & 0\\
%0 & 1 & 0\\
%0 & 0 & -1
%\end{bmatrix},\\
H_3  = \sum_{s,t}(r-t)E^*_{[s,t]}, \\
%\sim h_3 = \begin{bmatrix}
%1 & 0  & 0\\
%0 & 0 & 0\\
%0 & 0 & -1
%\end{bmatrix}, \text{where $r = d-s-t$},\\
%\end{align*}
%\begin{align*}
R_1 & = \sum_{s,t}E^*_{[s+1,t]}AE^*_{[s,t]},\;
% \sim e_{-s} = \begin{bmatrix}
% 0 &  0 & 0\\
%1 &  0 & 0\\
%0 & 0 & 0
%\end{bmatrix},\\
R_2  = \sum_{s,t}E^*_{[s-1,t+1]}AE^*_{[s,t]}, \;
% \sim e_{-t} = \begin{bmatrix}
%0 & 0 & 0\\
%0 & 0 & 0 \\
%0 & 1 & 0
%\end{bmatrix},\\
R_3  = \sum_{s,t}E^*_{[s,t+1]}A^\top E^*_{[s,t]},\\
%\sim e_{-s-t} = \begin{bmatrix}
%0 & 0 & 0\\
%0 & 0 & 0 \\
%1 & 0 & 0
%\end{bmatrix},
%\end{align*}
%\begin{align*}
L_1 & = \sum_{s,t}E^*_{[s-1,t]}A^\top E^*_{[s,t]},\;
% \sim e_s = \begin{bmatrix}
%0 & 1 & 0\\
%0 & 0 & 0\\
%0 & 0 & 0
%\end{bmatrix},\\
L_2  = \sum_{s,t}E^*_{[s+1,t-1]}A^\top E^*_{[s,t]},\;
% \sim e_t = \begin{bmatrix}
%0 & 0 & 0 \\
%0 & 0 & 1 \\
%0 & 0 & 0
%\end{bmatrix},\\
L_3  = \sum_{s,t}E^*_{[s,t-1]}AE^*_{[s,t]}.
% \sim e_{-s-t} = \begin{bmatrix}
%0 & 0 & 1\\
%0 & 0 & 0\\
%0 & 0 & 0
%\end{bmatrix}.
\end{align*}

In the following Lemmas 2.1--2.3, we show the relations among these matrices: 
\begin{lemma}
We have the following.
\begin{itemize}
\item[{\rm (i)}] $A = R_1 + R_2 + L_3$ and $A^\top = L_1 + L_2 + R_3$.
\item[{\rm (ii)}] Let $\tilde{A} = A+A^\top$. Then $\tilde{A}$ is the adjacency matrix of $H(d,3)$.
\item[{\rm (iii)}] $\tilde{R} = R_1+R_3$ is the raising operator, $\tilde{F} = R_2 + L_2$ the flat operator, and  $\tilde{L} = L_1+L_3$ the lowering operator of $H(d,3)$.
\end{itemize}
\end{lemma}
\begin{proof}
Straightforward. See \cite[Chapter 30]{lecturenote} for the raising, flat, and lowering operators.
\end{proof}

\begin{lemma}
The following hold.
\begin{itemize}
\item[{\rm (i)}] $[L_1, R_1] = H_1$, $[H_1,L_1]= 2L_1$, and $[H_1,R_1] = -2R_1$.
\item[{\rm (ii)}] $[L_2, R_2] = H_2$, $[H_2,L_2]= 2L_2$, and $[H_2,R_2] = -2R_2$.
\item[{\rm (iii)}] $[L_3, R_3] = H_3$, $[H_3,L_3]= 2L_3$, and $[H_3,R_3] = -2R_3$.
\end{itemize}
\end{lemma}
\begin{proof}
{\rm (i)} Let $\hat{y} \in V$ such that $y = (0, \ldots, 0, 1, \ldots, 1, 2, \ldots, 2)$ with $r(y) = r, s(y) = y, t(y) = t$,  i.e., $\text{type}(\hat{y}) = (r,s,t)$.
\begin{align*}
[L_1, R_1]\hat{y} & = (L_1R_1 - R_1L_1)\hat{y}\\
& = (E^*_{[s,t]}A^\top E^*_{[s+1,t]}AE^*_{[s,t]}-E^*_{[s,t]}AE^*_{[s-1,t]}A^\top E^*_{[s,t]})\hat{y}\\
& = (r-s)\hat{y}.
\end{align*}
Hence, $[L_1, R_1] = H_1$.
\begin{align*}
[H_1,L_1]\hat{y} & = (H_1L_1-L_1H_1)\hat{y}\\
& = ((r+1)-(s-1))E^*_{[s-1,t]}A^\top E^*_{[s,t]}\hat{y} - (r-s)E^*_{[s-1,t]}A^\top E^*_{[s,t]}\hat{y}\\
& = 2L_1\hat{y}.
\end{align*}
Hence, $[H_1,L_1] = 2L_1$.
$$[H_1, R_1] = [H_1^\top, L_1^\top] \\
 = H_1^\top L_1^\top - L_1^\top H_1^\top\\
 = [L_1,H_1]^\top\\
 = -[H_1,L_1]^\top\\
 = -2R_1.$$
{\rm (ii)} and {\rm (iii)} are similar. 
\end{proof}

\begin{lemma}
In addition to the formulas in the previous lemma, the following hold.
\begin{itemize}
\item[{\rm (i)}] $H_3 = H_1 + H_2$, and $[H_i,H_j] = O$ for all $i,j\in \{1,2,3\}$.
\item[{\rm (ii)}] $[L_1, L_2] = L_3$.
\item[{\rm (iii)}] $[L_1, R_2] = O$.
\item[{\rm (iv)}] $[H_1,L_2] = -L_2$, $[H_2,L_1] = -L_1$.
\item[{\rm (v)}] $[L_1,L_3] = [L_2,L_3] = O$.
\end{itemize}
\end{lemma}
\begin{proof}
{\rm (i)} is trivial. Note that $H_1$, $H_2$ and $H_3$ are diagonal matrices.

\noindent
{\rm (ii)} Let $\hat{y} \in V$ be $y = (0, \ldots, 0, 1, \ldots, 1, 2, \ldots, 2)$ with $r(y) = r, s(y) = y, t(y) = t$, i.e., $\text{type}(\hat{y}) = (r,s,t)$.
\begin{align*}
[L_1, L_2]\hat{y} &= (E^*_{[s,t-1]}A^\top E^*_{[s+1,t-1]}A^\top E^*_{[s,t]} - E^*_{[s,t-1]}A^\top E^*_{[s-1,t]}A^\top E^*_{[s,t]})\hat{y}\\
& = E^*_{[s,t-1]}AE^*_{[s,t]}\hat{y}\\
& = L_3\hat{y}.
\end{align*}

\medskip\noindent
{\rm (iii)} Since $L_1$ changes each one of ones to zero and takes the sum, 
and 
$R_2$ changes one of ones to two and takes a sum, these two operations commute.

\medskip\noindent
{\rm (iv)} Since $L_2$ changes one of two to one and takes a sum while $H_1$ takes the difference between the number of zeros and ones, such number decreases by 1 between $H_1L_2$ and $L_2H_1$. The latter is similar. 

\medskip\noindent
{\rm (v)} Since $L_1$ changes each one of ones to zero and takes the sum, and 
$L_3$ changes one of twos to zero and takes the sum, these two operations commute. The latter is similar.
\end{proof}

\begin{cor}\label{cor:sl3c}
Suppose $H_1, H_2, L_1, L_2\in \xmat$ and $H_1$, $H_2$ are diagonal matrices. If $R_1 = L_1^\top$ and $R_2 = L_2^\top$, the following hold.
\begin{itemize}
\item[{\rm (ii)}] $[L_i, R_j] = \delta_{i,j}H_i$.
\item[{\rm (iii)}] $[H_1,L_1] = 2L_1$, $[H_1,L_2] = -L_2$, $[H_2,L_1] = -L_1$, $[H_2,L_2] = 2L_2$.
\item[{\rm (iv)}] $[L_1,[L_1,L_2]] = [L_2,[L_2,L_1]] = O$.
\end{itemize}
%Then the Lie algebra $\cL = \langle H_1, H_2, L_1, L_2, R_1, R_2\rangle$ is 
%isomorphic to $\mathit{sl}_3(\bC)$, and the Lie algebras $\cL_1 = \langle H_1, L_1, R_1\rangle$, $\cL_2 = \langle H_2, L_2, R_2\rangle$ and $\cL_3 = \langle H_3, L_3, R_3\rangle$ are all isomorphic to $\mathit{sl}_2(\bC)$.
\end{cor}

The following gives a characterization of 
the Lie algebra $\mathit{sl}_3(\bC)$: 
\begin{prop}[{\cite[Proposition in page 96 and Theorem (Serre) on page 99]{jeh}}]\label{prop:sl3c}
Let $\cL$ be the Lie algebra generated by $h_1, h_2, x_1, x_2, y_1, y_2$, and let 
$$C = \begin{bmatrix} 2 & -1\\ -1 & 2\end{bmatrix}.$$ 
Suppose for all $i, j\in \{1,2\}$, the following hold.
\begin{itemize} % numbering error suggested by KU
\item[{\rm (i)}] $[h_i, h_j] = 0$.
\item[{\rm (ii)}] $[x_i, y_j] = \delta_{i,j}h_i$.
\item[{\rm (iii)}] $[h_i,x_j] = c_{i,j}x_j$.
\item[{\rm (iv)}] $[h_i,y_j] = -c_{j,i}y_j$.
\item[{\rm (v)}] $(ad x_i)^{-C_{j,i}+1}(x_j) = 0$, if $i\neq j$, i.e., $[x_1,[x_1,x_2]] = [x_2,[x_2,x_1]] = 0$.
\item[{\rm (vi)}] $(ad y_i)^{-C_{j,i}+1}(y_j) = 0$, if $i\neq j$, i.e., $[y_1,[y_1,y_2]] = [y_2,[y_2,y_1]] = 0$.
\end{itemize}
Then the Lie algebra $\cL = \langle h_1, h_2, x_1, x_2, y_1, y_2\rangle$ is 
isomorphic to $\mathit{sl}_3(\bC)$, and the Lie algebras $\langle h_1, x_1, y_1\rangle$, $\langle h_2, x_2, y_2\rangle$ and $\langle h_3, x_3, y_3\rangle$ are all isomorphic to $\mathit{sl}_2(\bC)$, where $h_3 = h_1 + h_2$, $x_3 = [x_1,x_2]$ and $y_3 = -[y_1,y_2]$.
\end{prop}
%\begin{proof}
%This is due to Serre. See \cite[Theorem (Serre) on page 99]{jeh}.
%\end{proof}

Then, comparing Lemmas 2.1--2.3, Corollary 2.4, and Proposition 2.5, we obtain the following: 
\begin{prop}
The following hold.
\begin{itemize}
\item[{\rm (i)}] The Lie algebra generated by $H_1, H_2, L_1, L_2, R_1, R_2$ is isomorphic to $\mathit{sl}_3(\bC)$ with a Cartan subalgebra generated by $H_1$ and $H_2$, and a Borel subalgebra generated by $H_1, H_2, L_1, L_2, L_3$.
\item[{\rm (ii)}] The Lie algebra generated by the triple $H_1, L_1,R_1$ (resp. the triple $H_2, L_2,R_2$ and $H_3, L_3,R_3$) is isomorphic to $\mathit{sl}_2(\bC)$ with a Cartan subalgebra generated by $H_1$ (resp. $H_2$, $H_3$) and a Borel subalgebra generated by $H_1, L_1$ (resp. $H_2, L_2$, and $H_3, L_3$).
\end{itemize}
\end{prop}

\note
The triple $H_1, L_1,R_1$ comes from the Terwilliger algebra of $H(d,\{0,1\})$, the triple $H_2, L_2,R_2$ from that of $H(d,\{1,2\})$, and the triple $H_3, L_3,R_3$ from that of $H(d,\{0,2\})$.

\subsection{Higher order relations in the Lie algebra $sl_3(\CC)$}

In this subsection, we present higher-order relations 
in the Lie algebra $\mathit{sl}_3(\bC)$. 
%based on the fundamental commutation relations and 
%the algebraic structure introduced earlier. 
%These higher-order relations reflect deeper structural properties of the Lie algebra and provide a framework for understanding the action of the Terwilliger algebra.
In the following we write $xy-yx = [x,y] = \mathit{ad}(x)(y)$.

The following is a fundamental commutation relations: 
\begin{lemma}[{\cite[Lemma in page 116]{jeh}}]\label{lemma:jeh}
Let $x$ and $y$ be elements of an associative algebra. Then 
$$[y^\ell,z] = \sum_{m=1}^\ell \binom{\ell}{m}(\mathit{ad}(y))^m(z)y^{\ell-m}.$$
\end{lemma}
%\begin{proof}
%By induction. For $\ell = 1$ case is trivial, as $[y,z] = (\mathit{ad}(y))(z)$.% Suppose it holds for $\ell$. Then, by using the induction hypothesis, we have
%\begin{align*}
%[y^{\ell+1},z] & = y^{\ell+1}z - zy^{\ell+1}\\
%& = y\cdot y^\ell z - y\cdot zy^\ell - y^\ell z \cdot y + zy^\ell\cdot y  + y^\%ell z \cdot y - zy^\ell \cdot y  + yz\cdot y^\ell -zy\cdot y^{\ell}\\
%& =(\mathit{ad}(y)) [y^\ell,z] + [y^\ell,z]\cdot y + (\mathit{ad}(y))(z)y^\ell\%\ % as suggested by KU added the last term
%& = \sum_{m=1}^\ell\binom{\ell}{m}(\mathit{ad}(y))^{m+1}(z)y^{\ell-m} + \sum_{m%=1}^\ell \binom{\ell}{m}(\mathit{ad}(y))^m(z)y^{\ell-m+1} \\
%& \qquad + (\mathit{ad}(y))(z)y^\ell\\
%& = \sum_{m=1}^{\ell+1}\binom{\ell}{m-1}(\mathit{ad}(y))^{m}(z)y^{\ell-m+1} + \%sum_{m=1}^\ell \binom{\ell}{m}(\mathit{ad}(y))^m(z)y^{\ell-m+1} \\ % m=2 to 1
%& = \sum_{m=1}^{\ell+1}\binom{\ell+1}{m}(\mathit{ad}(y))^m(z)y^{\ell+1-m}.
%\end{align*}
%Thus, the formula holds.
%\end{proof}

Using Lemma \ref{lemma:jeh}, we have the following: 
\begin{lemma}\label{lemma:commutator-formula}
Let $i, j, k$ be positive integers. Then, the following hold.
\begin{itemize}
\item[{\rm (i)}] $(\mathit{ad}(R_2))^k(R_1^i) =  i\cdot (i-1)\cdots (i-k+1)R^k_3R_1^{i-k}$.
\item[{\rm (ii)}] $(\mathit{ad}(R_1))^k(R_2^j) =  (-1)^kj\cdot (j-1)\cdots (j-k+1)R^k_3R_2^{j-k}$.
\item[{\rm (iii)}] ${\displaystyle [R_2^j,R_1^i] = \sum_{k=1}^j k! \binom{i}{k}\binom{j}{k}R^k_3R^{i-k}_1R^{j-k}_2}$.
\item[{\rm (iv)}] ${\displaystyle [R_1^i,R_2^j] = \sum_{k=1}^i (-1)^kk!\binom{i}{k}\binom{j}{k}R_3^kR_2^{j-k}R_1^{i-k}}$.
\end{itemize}
\end{lemma}
\begin{proof}
Observe that $[R_1,R_2] = -R_3$ and $[R_1,[R_1,R_2]] = [R_1, -R_3] = O = [R_2,R_3]$.

(i) For $k = 1$, we apply the previous lemma. Then, 
$$\mathit{ad}(R_2)(R_1^i) = -[R_1^i, R_2] = iR_3R_1^{i-1},$$
and the formula holds in this case.

Now, using the induction hypothesis, we have
\begin{align*}
(\mathit{ad}(R_2))^k(R_1^i) & = \mathit{ad}(R_2)(i\cdot (i-1)\cdots (i-k+2)R^{k-1}_3R_1^{i-k+1})\\
& = i\cdot (i-1)\cdots (i-k+2)R^{k-1}_3\mathit{ad}(R_2)(R_1^{i-k+1})\\
& = i\cdot (i-1)\cdots (i-k+1)R^k_3R_1^{i-k}.
\end{align*}

(ii) Similarly, for $k = 1$, we apply the previous lemma. Then, 
$$\mathit{ad}(R_1)(R_2^j) = -[R_2^j, R_1] = -jR_3R_2^{j-1},$$
and the formula holds in this case.

Now, using the induction hypothesis, we have
\begin{align*}
(\mathit{ad}(R_1))^k(R_2^j) & = (-1)^{k-1}\mathit{ad}(R_1)(j\cdot (j-1)\cdots (j-k+2)R^{k-1}_3R_2^{j-k+1})\\
& = (-1)^{k-1}j\cdot (j-1)\cdots (j-k+2)R_3^{k-1}\mathit{ad}(R_1)(R_2^{j-k+1})\\
& = (-1)^{k}j\cdot (j-1)\cdots (j-k+1)R^k_3R_2^{j-k}.
\end{align*}

{\rm (iii)} First, we apply the previous lemma and then {\rm (i)}.
\begin{align*}
[R_2^j,R_1^i] &= \sum_{k=1}^j \binom{j}{k}(\mathit{ad}(R_2))^k(R_1^i)R^{j-k}_2\\
& = \sum_{k=1}^j i\cdot (i-1)\cdots (i-k+1)\binom{j}{k}R^k_3R^{i-k}_1R^{j-k}_2\\
& = \sum_{k=1}^j k! \binom{i}{k}\binom{j}{k}R^k_3R^{i-k}_1R^{j-k}_2.
\end{align*}

We have the formula.

{\rm (iv)} Similarly, first, we apply the previous lemma and then {\rm (ii)}.
\begin{align*}
[R_1^i,R_2^j] &= \sum_{k=1}^i \binom{i}{k}(\mathit{ad}(R_1))^k(R_2^j)R^{i-k}_1\\
& = \sum_{k=1}^i (-1)^{k}j\cdot (j-1)\cdots (j-k+1)\binom{i}{k}R^k_3R^{j-k}_2R^{i-k}_1\\
& = \sum_{k=1}^i (-1)^kk!\binom{i}{k}\binom{j}{k}R_3^kR_2^{j-k}R_1^{i-k}.
\end{align*}

We have the formula.
\end{proof}

The following will be used in Section 4: 
%\newpage
\begin{lemma}\label{lemma:commutators}
Let $V$ be the standard module and suppose $\bv\neq \bo$ satisfies $L_1\bv = L_2\bv = \bo$, $H_1\bv = m_1\bv$ and $H_2\bv = m_2\bv$. Then, the following hold.
\begin{itemize}
%\item[{\rm (i)}] $L_2R_1^i \bv = L_3R_1^i\bv = L_1R_2^j \bv = L_3R_2^j\bv = \bo$ for every nonnegative integer $i$ and $j$.
\item[{\rm (i)}] $L_3R_1^iR_2^j \bv  = -ij(m_2 - j+1)R_1^{i-1}R_2^{j-1}\bv \in \spn(R_1^{i-1}R_2^{j-1}\bv)$ for every integers $i$ and $j$.
\item[{\rm (ii)}] $L_3R_2^jR_1^i \bv  =  ij( m_1 -i+1)R_2^{j-1}R_1^{i-1}\bv\in \spn(R_2^{j-1}R_1^{i-1}\bv)$ for every integers $i$ and $j$.
\end{itemize}
\end{lemma}
\begin{proof}
We first treat the case when $i=0$ or $j = 0$. $L_2R_1^i \bv = L_1R_2^j \bv = \bo$ is clear as $[L_2,R_1] = [L_1,R_2] = O$ and $L_2\bv = L_1\bv =\bo$. $L_3\bv =\bo$ as $[L_1,L_2] = L_3$.

Observe that  $[R_1,L_3] = L_2$, $[R_1,[R_1,L_3]] = O$, $[R_2, L_3] = -L_1$, $[R_2, [R_2,L_3]] = O$ and $L_2R_1^i\bv = L_1R_2^j \bv = \bo$ for every nonnegative integers $i$ and $j$, by Lemma~\ref{lemma:jeh} we have
$[L_3, R_1^i] = -i L_2R_1^{i-1} = -i R_1^{i-1}L_2$ and $[L_3, R_2^j] = jL_1R_2^{j-1} = j R_2^{j-1}L_1$. Hence, 
$$L_3R_1^i\bv = R_1^iL_3\bv + [L_3,R_1^i]\bv = \bo -i L_2R_1^{i-1}\bv = -iR_1^{i-1}L_2 \bv = \bo,$$
and
$$L_3R_2^j\bv = R_2^jL_3\bv + [L_3,R_2^j]\bv = \bo + j L_1R_2^{j-1}\bv = jR_2^{j-1}L_1 \bv = \bo.$$

We now assume that both $i$ and $j$ are positive.

(i) Since $[R_1, L_3] = L_2$, $[R_1, [R_1,L_3]] = O$ and $[L_1, R_2] = O$, by Lemma~\ref{lemma:jeh} we have
$[L_3, R_1^i] = -iL_2R_1^{i-1} = -i R_1^{i-1}L_2$.

Recalling that $[R_2,L_2] = -H_2$, $[R_2,-H_2] = -2R_2$ and $[R_2,[R_2,-H_2]] = O$, by Lemma~\ref{lemma:jeh}  we have
$[R_2^j, H_2] = j[R_2,H_2]R_2^{j-1} = 2jR_2^j$,  and hence,
\begin{align*}
[R_2^j, L_2] & = j[R_2,L_2]R_2^{j-1} + \binom{j}{2}[R_2,[R_2,L_2]]R_2^{j-2} \\
& = -jH_2R_2^{j-1} - j(j-1)R_2^{j-1}\\
& = -j R_2^{j-1}H_2 +2j(j-1)R_2^{j-1} - j(j-1)R_2^{j-1}\\
& = -j R_2^{j-1}H_2 + j(j-1)R_2^{j-1},
\end{align*}
we have by (i),
\begin{align*}
L_3R_1^iR_2^j \bv & = R_1^iL_3R_2^j \bv + [L_3,R_1^i]R_2^j\bv\\
& = \bo  -i R_1^{i-1}L_2R_2^j\bv\\
& = -i R_1^{i-1}R_2^jL_2\bv + i R_1^{i-1}[R_2^j,L_2]\bv \\
& = \bo - iR_1^{i-1} (-j R_2^{j-1}H_2 + j(j-1)R_2^{j-1})\bv\\
& = ij(m_2 - j+1)R_1^{i-1}R_2^{j-1}\bv,
\end{align*}
as desired.

(ii) Since $[R_2, L_3] = -L_1$, $[R_2, [R_2,L_3]] = O$ and $[L_1, R_2] = O$, by Lemma~\ref{lemma:jeh} we have
$[L_3, R_2^j] = jL_1R_2^{j-1} = j R_2^{j-1}L_1$.

Recalling that $[R_1,L_1] = -H_1$, $[R_1,-H_1] = -2R_1$ and $[R_1,[R_1,-H_1]] = O$, by Lemma~\ref{lemma:jeh} we have
$[R_1^i, H_1] = i[R_1,H_1]R_1^{i-1} = 2iR_1^i$,  and hence,
\begin{align*}
[R_1^i, L_1] & = i[R_1,L_1]R_1^{i-1} + \binom{i}{2}[R_1,[R_1,L_1]]R_2^{i-2} \\
& = -iH_1R_1^{i-1} - i(i-1)R_1^{i-1}\\
& = -i R_1^{i-1}H_1 +2i(i-1)R_1^{i-1} - i(i-1)R_1^{i-1}\\
& = -i R_1^{i-1}H_1 + i(i-1)R_1^{i-1},
\end{align*}
we have by (i),
\begin{align*}
L_3R_2^jR_1^i \bv & = R_2^jL_3R_1^i \bv + [L_3,R_2^j]R_1^i\bv\\
& = \bo  +j R_2^{j-1}L_1R_1^i\bv\\
& = j R_2^{j-1}R_1^iL_1\bv - j R_2^{j-1}[R_1^i,L_1]\bv \\
& = \bo + jR_2^{j-1} (-i R_1^{i-1}H_1 + i(i-1)R_1^{i-1})\bv\\
& = -ij( m_1 -i+1)R_2^{j-1}R_1^{i-1}\bv,
\end{align*}
as desired.

\end{proof}

\section{Algebras and Modules}

In this section, we give some properties of a highest weight vector and module. Firstly, we give an economical generator of $\mathcal{T}$. 
\begin{prop}
The following hold.
$\bC[A] = \bC[A^\top] = \bC[A, A^\top]$ is the Bose-Mesner algebra of $H^*(d,3)$.
\end{prop}
\begin{proof}
Since the matrices $A$ and $A^\top$ are commuting normal matrices, they are simultaneously diagonalizable, having the same eigenvalues with multiplicities. Hence, we have the assertions.
\end{proof}

Then we have the following: 
\begin{prop}
The following hold.
\begin{itemize}
\item[{\rm (i)}] $\cT = \cT(x) = \langle A, E^*_{i,j} \mid (i,j)\in \Delta\rangle$.
\item[{\rm (ii)}] $\cT(x) = \langle L_1, L_2, R_1, R_2\rangle$.
\item[{\rm (iii)}] Let $\cL$ be the Lie algebra generated by $L_1, L_2, R_1, R_2$. Then $\cL\simeq \mathit{sl}_3(\bC)$, and $\cL$ submodules of $V$ are $\cT$ submodules and vice versa. Moreover, an $\cL$-submodule $W$ of $V$ is irreducible if and only if a $\cT$-module $W$ of $V$ is irreducible.
\end{itemize}
\end{prop}
\begin{proof}
Clearly, 
$$\langle L_1, L_2, R_1, R_2\rangle \subseteq \langle A, E^*_{i,j} \mid (i,j)\in \Delta\rangle.$$
Since $[L_1, R_1] = H_1$, $[L_2, R_2] = H_2$, $[L_1, L_2] = L_3$ and $[R_1, R_2] = R_3$, 
$$A = R_1 + R_2 + L_3, \; A^\top = L_1 + L_2 + R_3 \in \langle L_1, L_2, R_1, R_2\rangle.$$
Moreover, since $H_1, H_2\in \langle L_1, L_2, R_1, R_2\rangle$, each $E^*_{i,j}$ belongs to $\langle L_1, L_2, R_1, R_2\rangle$ and the assertions hold by Corollary~\ref{cor:sl3c}.
\end{proof}

\begin{comment}

\begin{note}
The Terwilliger algebra is semisimple, as it is closed under the transpose and the complex conjugation, and hence, every indecomposable module is irreducible. 
\end{note}

\begin{note}
The Lie algebra $\cL = \mathit{sl}_3(\bC)$ (resp. $\mathit{sl}_2(\bC)$) is the set of $3\times 3$ (resp. $2\times 2$) trace zero matrices over $\bC$ with the Lie bracket defined by $[x,y] = xy-yx$. It is a simple Lie algebra generated by 
$$
e_s = \begin{bmatrix}
0 & 1 & 0\\
0 & 0 & 0\\
0 & 0 & 0
\end{bmatrix},\;
e_t = \begin{bmatrix}
0 & 0 & 0 \\
0 & 0 & 1 \\
0 & 0 & 0
\end{bmatrix}, \;
e_{-s} = \begin{bmatrix}
0 & 0 & 0 \\
1 & 0 & 0\\
0 & 0 & 0
\end{bmatrix}, \; 
e_{-t} = \begin{bmatrix}
0 & 0 & 0\\
0 & 0 & 0 \\
0 & 1 & 0
\end{bmatrix}.
$$
It has the following basis.
$$\{e_s, e_t, e_{s+t} = [e_s,e_t], e_{-s}, e_{-t}, e_{-s-t} = [e_{-s},e_{-t}], h_1 = [e_s,e_{-s}], h_2 = [e_t,e_{-t}]\}$$ 
with 
$$e_{s+t} = \begin{bmatrix}
0 & 0 & 1\\ % edited in 1.0.2
0 & 0 & 0 \\
0 & 0 & 0
\end{bmatrix}, e_{-s-t} = \begin{bmatrix}
0 & 0 & 0\\
0 & 0 & 0 \\
1 & 0 & 0
\end{bmatrix}, h_1 = \begin{bmatrix}
1 & 0 & 0\\
0 &  -1 & 0 \\
0 & 0 & 0
\end{bmatrix}, h_2 =  \begin{bmatrix}
0 & 0 & 0 \\
0 & 1 & 0 \\
0 & 0 & -1
\end{bmatrix}.$$
\end{note}

\end{comment}

It is well-known that all the irreducible modules can be viewed as the following cyclic module: 
\begin{prop}[{\cite[Theorem on page 108]{jeh}}] \label{prop:pbw}
Let $\cL = \langle L_1, L_2, R_1, R_2\rangle$, and $W$ an irreducible submodule of $V$. Then there is one of the highest weight vectors $\bv\in W$ satisfying $L_1\bv = L_2\bv = \bo \neq \bv$, determined up to a nonzero scalar multiple,  and $W$ is spanned by either one of the following set of vectors
$$\{R_3^{k}R_2^{j}R_1^{i}\bv\mid i, j, k\geq 0\}, \text{ and } \{R_3^{k}R_1^{j}R_2^{i}\bv\mid i, j, k\geq 0\}.$$
\end{prop}
\begin{proof}
The module $W$ can be regarded as a standard cyclic $\cL$ module; we have the assertion from the theorem. 
\end{proof}

\begin{defin}
Let $\bv\in V$ be a nonzero vector of $V$ satisfying $H_1\bv = m_1\bv$ and $H_2\bv = m_2\bv$. Then we call $\lambda = (m_1, m_2)$ the weight of $\bv$.
\end{defin}

\begin{lemma}
Let $W$ be an irreducible $\cL$ submodule of $V$ with the highest weight vector $\bv = E^*_{{[s,t]}}\bv$ with weight $\lambda = (m_1, m_2)$. Set $r = d-s-t$, $\bv_{i,j,k} = R_3^{k}R_2^{j}R_1^{i}\bv$, and $\bw_{j,i,k} = R_3^{k}R_1^{i}R_2^{j}\bv$. Then,  $L_1\bv = L_2\bv = L_3\bv = \bo$, $m_1 = r-s$, $m_2 = s-t$, and the following hold.
\begin{itemize}
\item[{\rm (i)}] The weight of $\bv_{i, j, k}$ and $\bw_{j,i,k}$ is 
$(m_1 -2i+j-k, m_2+i-2j-k)$ if it is not a zero vector.% = (r-s-2i+j-k, s-t+i-2j-k)$.
\item[{\rm (ii)}] $\bv_{i,j,0} \neq \bo$ if $0\leq i\leq m_1$, $0\leq j\leq m_2+i$, and $\bw_{j,i,0} \neq \bo$ if $0\leq j\leq m_2$, $0\leq i\leq m_1+i$.
\end{itemize}
\end{lemma}
\begin{proof}
Recall that $L_1\bv = \bo$, $L_2\bv = \bo$ and hence $L_3\bv = [L_1, L_2]\bv = \bo$. 

Since $\bv = E^*_{{[s,t]}}\bv$, by the definitions of $H_1$ and $H_2$,$H_1\bv = (r-s)\bv$ and $H_2\bv = (s-t)\bv$. % and $H_3\bv = (H_1+H_2)\bv = (r-t)\bv$. 
Hence, $m_1 = r-s$, $m_2 = s-t$.

(i) Suppose $\bw =  E^*_{{[s',t']}}\bw \neq \bo$ and $r' = d-s'-t'$. Then, the weight of $\bw$ is $\mu = (r'-s', s'-t')$. 

Since $\text{type}(R_1\bw) = (r(R_1\bw), s(R_1\bw), t(R_1\bw)) = (r'-1, s'+1, t')$, the weight of $R_1\bw$ is $(r'-s'-2, s'-t'+1) = \mu + (-2,1)$. 

Similarly, $\text{type}(R_2\bw) = (r(R_2\bw), s(R_2\bw), t(R_2\bw)) = (r', s'-1, t'+1)$ and the weight of $R_2\bw$ is $(r'-s'+1, s'-t'-2) = \mu + (1,-2)$. 

Finally, $\text{type}(R_3\bw) = (r(R_3\bw), s(R_3\bw), t(R_3\bw)) = (r'-1, s', t'+1)$ and the weight of $R_3\bw$ is $(r'-s'-1, s'-t'-1) = \mu + (-1,-1)$. 

Therefore, we have (i) by induction.

(ii) Since $\cL_1\bv$ is an irreducible $\cL_1$ module in $\cL\bv$ of highest weight $m_1 = s-t$. Hence, its dimension is $m_1+1$, and $\bv, R_1\bv, \ldots, R_1^{m_1}\bv$ form a basis. $\bv_{i,0,0} = R_1^i\bv\neq \bo$ if $0\leq i\leq m_1$, and the weight of $\bv_{i,0,0}$ is $(m_1-2i, m_2+i)$. 

Since $[L_2,R_1] = O$, $L_2$ and $R_1$ commute. Hence, $L_2\bv_{i,0,0} = L_2R_1^i\bv = R_1^iL_2\bv = \bo$. Thus, $\bv_{i,0,0}$ generates an irreducible $\cL_2$ module in $\cL\bv$ of highest weight $m_2+i$, if $\bv_{i,0,0}\neq \bo$. Hence, its dimension is $m_2+i+1$, and $R_1^i\bv, R_2R_1^i\bv \ldots, R_2^{m_2+i}R_1^i\bv$ form a basis. Thus, $\bv_{i,j,0} = R_2^jR_1^i\bv\neq \bo$ if $0\leq i\leq m_1$, $0\leq j\leq m_2+i$.

Since $[L_1, R_2] = O$, the assertion for $\bw_{j,i,0} = R_1^iR_2^j\bv$ can be proved similarly.
\end{proof}

\begin{lemma}\label{lemma:parameters}
Suppose $\lambda = \lambda = (m_1,m_2)$ be the highest weight of an irreducible $\cL$ submodule $W$ of $V$, and $\bv = E^*_{{[s,t]}}\bv$ a highest weight vector. Set $r = d-s-t$. Then, the following hold.
\begin{itemize}
\item[{\rm (i)}] $(r,s,t) = (m_1 + m_2+t, m_2+t,t)$, and $r\geq s\geq t$.
\item[{\rm (ii)}] $t\in \{0, 1, \ldots, [\frac{d}{3}]\}$.
\item[{\rm (iii)}] $m_1 = r-s = d-3t-2m_2$.
\item[{\rm (iv)}] $m_2 = s-t  \in \{0, 1, \ldots, [\frac{d-3t}{2}]\}$.
\end{itemize}
\end{lemma}
\begin{proof}
Since $m_1= r-s\geq 0$, $m_2 = s-t \geq 0$ and $(r,s,t) = (m_1+m_2+t, m_2+t, t)$, $d = m_1 + 2m_2 + 3t$. Hence, $0\leq t = (d-m_1-2m_2)/3 \leq [\frac{d}{3}]$.

Moreover, since $2m_2 \leq m_1 + 2m_2 = d-3t$, $m_2 \leq \frac{d-3t}{2}$.
\end{proof}

\medskip
In the next section, we show that if integers $r, s, t$ satisfy $r \geq s\geq t$ with $d = r+s+t$, then there is an irreducible $\cL$ module with the highest weight vector of type $(r,s,t)$.  

We close this section with a technical lemma.

\begin{lemma} \label{lemma:technical}
The following hold for integers $i, j, k, i', j', k',p,q$.
\begin{itemize}
\item[{\rm (i)}] $(2i-j+k, -i+2j+k) = (2i'-j'+k', -i'+2j'+k')$ implies $(i,j,k) = (i',j',k')$ if $\min(i,j) = \min(i',j')$.
\item[{\rm (ii)}] $(2p-q, -p+2q) = (2i-j+k, -i+2j+k)$ implies $(i,j,k) = (p-t,q-t,t)$ if $\min(p,q) = \min(i,j) + t$. 
\end{itemize}
\end{lemma}
\begin{proof}
(i) By taking the differences of the entries, we have $3(i-j) = 3(i'-j')$. Hence $i-j = i'-j'$ and $i\leq j$ if and only if $i'\leq j'$. Since $\min(i,j) = \min(i',j')$, $i=i'$ or $j = j'$. Therefore, $i = i'$, $j = j'$ and $k = k'$.

(ii) Similarly, we have $p-q = i-j$, and  $\min(p,q) = \min(i,j) + t$. Therefore, $i = p-t$, $j = q-t$ and $k = t$.
\end{proof}

\section{Irreducible modules of $\cL$}

In this section, we give the structures of the Irreducible modules of $\cL$. 
Let $\cL = \langle  L_1, L_2, R_1, R_2\rangle$. Let $V_{[s,t]}$ be the subspace of $V$ spanned by vectors with $r$ zeros, $s$ ones and $t$ twos. Thus, 
$$V_{[s,t]} = E^*_{[s,t]}V.$$

Let $\bv\in V_{[s,t]}$ such that $L_1\bv = L_2\bv = \bo$. Since $H_1\bv = (r-s)\bv$ and $H_2\bv = (s-t)\bv$, $\bv$ generates an $\cL$ module with a highest weight vector $\bv$.

\begin{lemma}\label{lemma:key}
Let $\bv$ be one of the heigest weight vectors with weight $\lambda = (m_1, m_2)$ of an irreducible $\cL$ submodule $W$ of $V$. Let $\bv_{i,j,k} = R_3^{k}R_2^{j}R^{i}_1\bv$ for nonnegative integers $i,j,k$.
Let 
$$U_{\ell}  = \spn(\bv_{i,j,k} = R_3^{k}R_2^{j}R^{i}_1\bv\mid i,j,k\geq 0, \min(i,j) \leq \ell), \text{ and } U_{-1} = \bo$$
for each nonnegative integer $\ell$. Then, the following hold.
\begin{itemize}
\item[{\rm (i)}] $U_{\ell} =  \spn(R_3^{k}R^{i}_1R_2^{j}\bv\mid i,j,k\geq 0, \min(i,j) \leq \ell)$, and $U_\ell$ is an $\cL_3$ submodule of $V$. Moreover, $U_\ell$ is a direct sum of weight subspaces of $\cL$, and
$$U_{\ell}  = \spn(\bv_{i,j,k} = R_3^{k}R_2^{j}R^{i}_1\bv\mid i,j,k\geq 0, i\leq m_1, j\leq m_2, \min(i,j) \leq \ell).$$
\item[{\rm (ii)}] Let $\bv_{i,j,0} = R_2^{j}R^{i}_1\bv$ with $\min(i,j) = \ell$. Then, $L_3\bv_{i,j,0}\in U_{\ell-1}$. %and $\bv_{i,j,0} \not\in U_{\ell-1}$ 
\item[{\rm (iii)}] Let $W_{i,j} = \spn(\bv_{i,j,k} \mid 0\leq k\leq m_1+m_2-i-j)$. If $\bv_{i,j,0} \not\in U_{\ell-1}$ with $\ell = \min(i,j)$, $i\leq m_1$ and $j\leq m_2$, then $(W_{i,j} + U_{\ell-1})/U_{\ell-1}$ is an irreducible $\cL_3$ module with highest weight $m_1+m_2 -i-j$.
\item[{\rm (iv)}] The  set of vectors 
$$\mathcal{B} = \{\bv_{i,j,k} = R_3^{k}R_2^{j}R_1^{i}\bv \mid 0\leq i\leq m_1, 0\leq j\leq m_2, 0\leq k\leq m_1+m_2-i-j\}$$
generates $W$ as a linear space.
\end{itemize}
\end{lemma}
\begin{proof}
(i) Let $U_{\ell}' =  \spn(R_3^{k}R^{i}_1R_2^{j}\bv\mid i,j,k\geq 0, \min(i,j) \leq \ell)$. 

First, note that $R_1^i\bv \neq 0$ if $i\leq m_1$ and $R_2^j\bv\neq 0$ if $j\leq m_2$. By definition, $U_0 = U_0'$. Suppose $\ell = \min(i,j) \geq 1$. It suffices to show the following. 
$$R_2^{j}R^{i}_1\bv - R_1^iR_2^j\bv \in U_{\ell-1} \cap U_{\ell-1}'.$$%, \text{ and } R^{i}_1R_2^{j}\bv \in U_{\ell}.$$
By Lemma~\ref{lemma:commutator-formula} (iii), and (iv),
\begin{align*}
R_2^{j}R_1^{i}\bv - R_1^{i}R_2^{j}\bv & =  [R_2^j,R_1^i]\bv =  -[R_1^i,R_2^j]\bv\\
& = \sum_{k=1}^j k!\binom{i}{k}\binom{j}{k}R^k_3R^{i-k}_1R^{j-k}_2\bv \quad (\in U_{\ell-1}')\\
& = - \sum_{k=1}^i (-1)^kk!\binom{i}{k}\binom{j}{k}R_3^kR_2^{j-k}R_1^{i-k}\bv \quad (\in U_{\ell-1}).
\end{align*}
Hence, the difference is in $U_{\ell-1} \cap U_{\ell-1}'$.

(ii) We prove by induction on $\ell$. If $\ell = 0$, then $U_{\ell-1} = \bo$. By Lemma~\ref{lemma:commutators}, $L_3\bv_{i,0,0} = L_3\bv_{0,j,0} = \bo$. Moreover, as we have seen in the proof of (i), $\bv_{i,0,0} \neq 0$ if $i\leq m_1$ and $\bv_{0,j,0} \neq 0$ if $j\leq m_2$. 

(iii) Suppose $\ell = \min(i,j) >0$. Then, by Lemma~\ref{lemma:commutators} (ii),
$$L_3R_2^jR_1^i \bv  =  ij( m_1 -i+1)R_2^{j-1}R_1^{i-1}\bv\in U_{\ell-1}.$$
Suppose $0 < i\leq m_1$ and $0 < j\leq m_2$. Then, by induction hypothesis, $R_2^{j-1}R_1^{i-1}\bv\not\in U_{\ell-2}$. Hence, $L_3R_2^jR_1^i \bv \not\in U_{\ell-2}$. In particular, $R_2^jR_1^i \bv \not\in U_{\ell-2}$ with $\cL_3$-weight $m_1+m_2-i-j$.  Therefore, $(W_{i,j} + U_{\ell-1})/U_{\ell-1}$ is an irreducible $\cL_3$ of dimension $m_1+m_2-i-j+1$.

%Suppose $R_2^jR_1^i \bv \in U_{\ell-1}$. Since the weight of $R_2^jR_1^i \bv$ is $(m_1-2i+j, m_2+i-j)$, $R_2^jR_1^i \bv$ is a linear combination of vectors of the same weight in $U_{\ell-1}$. Since $R_2^jR_1^i \bv \not\in U_{\ell-2}$, $R_3^{k'}R_2^{j'}R_1^{i'}\bv$ with $\min(i',j') = \ell-1$ must appear in the sum. Since
%$$(m_1-2j+i, m_2+i-2j) = (m_1-2i'+j'-k', m_2+i'-2j'-k'),$$
%we have $(i',j',k') = (i-1,j-1,1)$ by Lemma~\ref{lemma:technical} (ii). Now,
%\begin{align*}
%L_3R_3R_2^{j-1}R_1^{i-1}\bv & = R_3L_3R_2^{j-1}R_1^{i-1}\bv + H_3R_2^{j-1}R_1^{i-1}\bv \\
%& = R_3L_3R_2^{j-1}R_1^{i-1}\bv + (m_1+m_2-i-j+2)R_2^{j-1}R_1^{i-1}\bv\\
%& \in (m_1+m_2-i-j+2)R_2^{j-1}R_1^{i-1}\bv + U_{\ell-2}.
%\end{align*}
%Comparing the weight as an $\cL_3$ module in $U_{\ell-1}/U_{\ell-2}$, we have $0> -ij( m_1 -i+1) = m_1+m_2-i-j+2 > 0$, a contradiction. 
%Suppose $\bv_{i,j,0} = R_2^jR_1^i \bv \not\in U_{\ell-1}$ with $\ell = \min(i,j)$, $i\leq m_1$ and $j\leq m_2$. Since both $U_\ell$ and $U_{\ell-1}$ are $\cL_3$ submodules, and $L_3\bv_{i,j,0} \in U_{\ell-1}$, $\bv_{i,j,0} + U_{\ell-1}$ is one of the highest weight vectors of $\cL_3$ in $U_{\ell}/U_{\ell-1}$. 
%
%(iii) Since $\bv_{i,j,0} + U_{\ell-1}$ is one of the $\cL_3$ highest weight vector with weight $m_1+m_2-i-j$. Hence, the assertion follows.

(iv) By (i), the following set of vectors 
$$\{\bv_{i,j,k} = R_3^{k}R_2^{j}R_1^{i}\bv \mid 0\leq i\leq m_1, 0\leq j\leq m_2, 0\leq k\}$$
generates $W$ as a linear space by Proposition~\ref{prop:pbw}.

We prove by induction on $0\leq \ell \leq \min(m_1, m_2)$. Assume $U_{\ell-1}$ is linearly spanned by the set 
$$\{\bv_{i,j,k}\mid 0\leq i \leq m_1, 0\leq j \leq m_2, 0\leq k \leq m_1+m_2 - i - j , \min(i,j) \leq \ell-1\}.$$
Since $U_{\ell-1}$ is an $\cL_3$ module by (i), if $\bv_{i,j,0} \in U_{\ell-1}$ for $\min(i,j) = \ell$, then $U_\ell$ is also spanned by the same set above. If a vector $\bv_{i,j,0} \not\in U_{\ell-1}$, then by (iii), we can restrict $k \leq m_1+m_2-i-j$. Therefore, $U_\ell$ is linearly spanned by the set
$$\{\bv_{i,j,k}\mid 0\leq i \leq m_1, 0\leq j \leq m_2, 0\leq k \leq m_1+m_2 - i - j , \min(i,j) \leq \ell\},$$
as desired.
%
%First note that the weight of $\bv_{i,j,k} = R_3^{k}R_2^{j}R_1^{i}\bv$ is $(m_1-2i+j-k, m_2+i-2j-k)$, and the weights of the set of vectors in the set
%$$\mathcal{B}_{\ell} = \{\bv_{i,j,k} = R_3^{k}R_2^{j}R_1^{i}\bv \mid 0\leq i\leq m_1, 0\leq j\leq m_2, 0\leq k\leq m_1+m_2-i-j, \ell = \min(i,j)\}$$
%are all distinct by Lemma~\ref{lemma:technical} (i). 
%Under the assumption, for each $0 \leq \ell \leq \min(m_1, m_2)$, the homomorphic image of $\mathcal{B}_{\ell}$ is linearly independent. Therefore, the set of vectors
%$$\bigcup_{\ell = 0}^{\min(m_1, m_2)} \mathcal{B}_{\ell} = \{\bv_{i,j,k} = R_3^{k}R_2^{j}R_1^{i}\bv \mid 0\leq i\leq m_1, 0\leq j\leq m_2, 0\leq k\leq m_1+m_2-i-j\}$$
%is linearly independent.
\end{proof}

We give the degree formula:
\begin{lemma}[{\cite[Corollary in page 139, see also the page 140]{jeh}}] \label{lem:deg}
If the highest weight of an irreducible module $W$ is $\lambda = (m_1, m_2)$, then 
$$\dim(W) = \frac{1}{2}(m_1+1)(m_2+1)(m_1+m_2+2).$$
\end{lemma}
%\begin{proof}
%See \cite[Corollary in page 139, see also the page 140]{jeh}. 
%\end{proof}

%\newpage
\begin{prop}\label{prop:basis}
Let $\bv$ be a highest weight vector with weight $\lambda = (m_1, m_2)$ of an irreducible $\cL$ submodule $W$ of $V$. Let $\bv_{i,j,k} = R_3^{k}R_2^{j}R^{i}_1\bv$. Then, the set
$$\mathcal{B} = \{\bv_{i,j,k} \mid 0\leq i\leq m_2, 0\leq j\leq m_1, 0\leq k\leq m_1+m_2-i-j\}$$
forms a basis of $W$, and $\dim W = \frac12(m_1+1)(m_2+1)(m_1+m_2+2)$.
\end{prop}
\begin{proof}
We now count the number of vectors in the set of vectors $\mathcal{B}$ using Lemma~\ref{lemma:key}.
\begin{align*}
\sum_{i=0}^{m_2}\sum_{j=0}^{m_1} (m_1+m_2 - i - j +1) & = 
\sum_{i=0}^{m_2}\sum_{j=0}^{m_1} (i + j +1)\\
& = \sum_{i=0}^{m_2}\frac{1}{2}(m_1+1)(m_1+2i+2)\\
& = \frac{1}{2}(m_1+1)(m_2+1)(m_1+m_2+2).
\end{align*}
Since the number of vectors in $\mathcal{B}$ is equal to the dimension of the module of $W$. Thus, we have the assertion.
\end{proof}

\begin{comment}
%\medskip
\note
The weight of the vector $\bv_{i,j,k}$ is $(m_1 - 2i+j-k, m_2+i-2j-k)$. By Lemma~\ref{lemma:technical} (ii), we know that the vectors in the following set have the same weight. 
$$\{\bv_{i-t,j-t,t}\mid 0\leq t\leq \min(i,j)\}.$$
It follows from Freudenthal's formula \cite[Theorem~22.3]{jeh} that the multiplicity of the weight $(m_1-2i+j,m_2+i-2j)$ is $\min(i,j)+1$, and we have that the set of vectors above is linearly independent, and an alternate proof of Proposition~\ref{prop:basis} follows. See Lemma~\ref{lemma:key} (iii).
 
%By Lemma~\ref{lem:deg}, we have the dimension of an irreducible module of a given highest weight. Hence, the linear independence above follows from it as far as the set of vectors generates the module.

\end{comment}

\begin{lemma}\label{lemma:existence}
For each $(r,s,t)$ with $r\geq s\geq t$, there is a highest weight vector $\bv$ such that $(r(\bv), s(\bv), t(\bv)) = (r,s,t)$. Every highest-weight vector is of this type.
\end{lemma}
\begin{proof}
By construction in Proposition~\ref{prop:basis}, a vector of type $(d,0,0)$ generates the principal module. 

Suppose $t = 0$, and we proceed by induction on $s$. Let $\bu$ be a vector in $\bC^{d-2}$ with $L_1'\bu = L_2'\bu = \bo$, where $L_1'$ and $L_2'$ are $L_1$ and $L_2$ operators for $H^*(d-2,3)$. Let $\bv = ((1,0)-(0,1))\ast\bu$, where $\ast$ denotes the concatenation. Then clearly, $L_1\bv = L_2\bv = \bo$. Therefore, the multiplicity of the weight of the vector of type $(r,s,0)$ in the irreducible module is at least one.

Suppose $t>0$. Then, by Lemma~\ref{lemma:parameters}, $r\geq s\geq t$. Let $\bu$ be the highest weight vector of the same highest weight of length $d-3$. 

Let $\bv = ((0,1,2) + (1,2,0)+(2,0,1)-(0,2,1)-(1,0,2)-(2,1,0))\ast \bu$, where $\ast$ denotes the concatenation of vectors. Then $L_1\bv = L_2\bv = \bo$ and $H_1\bv = m_1\bv$ and $H_2\bv = m_2\bv$ as desired.
\end{proof}

\medskip
In the proof above, since $\binom{d}{s} - \binom{d}{s-1}>0$ for $s \leq [\frac{d}{2}]$, there is a highest weight vector of type $(r,s,0)$ if $s \leq [\frac{d}{2}]$.

It is possible to prove the rest as well by subtracting all the factors of the irreducible modules with multiplicity with larger $(m_1, m_2)$ in lexicographical order, though it is very complicated.

\begin{example}
For a small $d$, we list $(r,s,t)$ of the highest weight vector and the highest weight $\lambda = (m_1,m_2)$ as $[r,s,t], (m_1, m_2): \dim, m$, where $\dim$ denotes the dimension of the corresponding irreducible module, and $m$ the multiplicity in the standard module in the following.
\begin{itemize}
\item $d=1$: $[1,0,0],(1,0):3,1$.
\item $d=2$: $[2,0,0],(2,0):6,1$, $[1,1,0],(0,1): 3,1$.
\item $d=3$: $[3,0,0],(3,0):10,1$, $[2,1,0],(1,1):8,2$, $[1,1,1],(0,0):1,1$
\item $d=4$:  $[4,0,0],(4,0):15,1$, $[3,1,0],(2,1):15,3$, $[2,2,0],(0,2):6,2$, $[2,1,1],(1,0):3,3$.
\item $d=5$: $[5,0,0],(5,0):21,1$, $[4,1,0],(3,1):24,4$, $[3,2,0],(1,2):15,5$, $[3,1,1],(2,0):6,6$, $[2,2,1],(0,1):3,5$.
\end{itemize}
\end{example}

\section{The structure of the Terwilliger algebra}

In this section, we prove the following result. 
%, which was first claimed by Mr. Uchida in his masterfs thesis. 

\begin{thm}\label{thm:uchida}
The following hold.
$$\cT(H^*(d,3)) \simeq \mathrm{Sym}^{(d)}(\mat_3(\bC))\simeq \bigoplus_{i\in \Lambda}\mat_i(\bC), $$
where 
\begin{align*}
\Lambda = &\left\{\frac12(d-3\ell-2m+1)(m+1)(d-3\ell-m+2)\:\right|\:\\ 
&\hspace{110pt}\left. 0\leq \ell\leq \left[\frac{d}{3}\right], 0\leq m\leq \left[\frac{d-3\ell}{2}\right]\right\}.
\end{align*}
\end{thm}

\begin{lemma}
Let $\cT_1$ be the Terwilliger algebra of $H^*(1,3)$. Then $\cT_1 \simeq \mat_3(\bC)$.
\end{lemma}
\begin{proof}
This is trivial. If $A^{(1)}$ is the adjacent matrix of $H^*(1,3)$, then ${A^{(1)}}^2 = {A^{(1)}}^\top$, and ${{E}^{(1)}}^*_{0,0}$, ${{E}^{(1)}}^*_{1,2}$ and ${{E}^{(1)}}^*_{2,1}$ are diagonal matrix units $e_{1,1}$, $e_{2,2}$ and $e_{3,3}$. Hence, every matrix unit is in $\cT_1$. 
\end{proof}

\medskip
Note that if the principal module of dimension $n$ is the only irreducible algebra module, it is isomorphic to $\nmat$.

\medskip
In the rest of this section, for  $i,j\in \{1,2,3\}$, $e_{i,j}$ denotes a matrix unit in $\mat_3(\bC)$, i.e., a three by three matrix with one in the $(i,j)$ entry and zero otherwise.

\medskip
For $r\leq d$, and given matrices $X_1, X_2, \ldots, X_\ell \in \nmat$ and positive integers $r_1, \ldots, r_\ell$ with $r_1 + \cdots + r_\ell \leq d$, let $L_{r_1, \ldots, r_\ell}(X_1, X_2, \ldots, X_\ell)$ denote the symmetrization of 
$$Y_1\otimes Y_2 \otimes \cdots \otimes Y_d \in \mat_{dn}(\bC),$$
where 
$$Y_i = \begin{cases} I & \text{if } 1\leq i \leq r_0,\\
X_j & \text{if } r_0 + \cdots + r_{j-1} < i \leq r_0 + r_1 + \cdots + r_j,\end{cases}$$
and $r_0 = d-(r_1+\cdots + r_\ell)$. Set $L(X) = L_1(X)$.

In particular, $L(I_n) = I_{dn}$, and for $Z\in \nmat$,  $L(Z) = L_d(Z)$ and
\begin{align*} 
L_d(Z) &= Z\otimes I \otimes \cdots \otimes I + I\otimes Z \otimes \cdots \otimes I + \cdots + I\otimes I \otimes \cdots \otimes Z\\
& \in \text{Sym}^{(d)}(\nmat) \subseteq \mat_{nd}(\bC).
\end{align*}

Let $A^{(m)}$ denote the adjacency matrix of $H^*(m,3)$ in $\mat_{3m}(\bC)$. Similarly, $X^{(m)}$ denotes the object of $H^*(m,3)$ in $\mat_{3m}(\bC)$ corresponding to $X$ for $H^*(d,3)$.  In our case, $A = A^{(d)} = L(A^{(1)})$ is the adjacency matrix of $H^*(d,3)$.

\begin{lemma}\label{lemma:a,eij}
The following holds.
\begin{itemize}
\item[{\rm (i)}] $A =A^{(d)} = L(A^{(1)})$ and $A^\top ={A^{(d)}}^\top = L({A^{(1)}}^\top)$.
\item[{\rm (ii)}] $E^*_{[s,t]} = L_{r,s,t}({E^*_{[0,0]}}^{(1)},{E^*_{[1,2]}}^{(1)},{E^*_{[2,1]}}^{(1)})$, where $r = d-s-t$.
\end{itemize}
\end{lemma}
\begin{proof}
{\rm (i)} This is clear from the definition of $H^*(d,3)$.

\medskip
{\rm (ii)} First we note that ${E^*_{[0,0]}}^{(1)} = e_{1,1}$, ${E^*_{[1,2]}}^{(1)} = e_{2,2}$ and ${E^*_{[2,1]}}^{(1)} = e_{3,3}$.

By definition, $E^*_{[0,0]} = e_{1,1}\otimes e_{1,1}\otimes \cdots \otimes e_{1,1} = L_{d}(e_{1,1}) = L_d({E^*_{[0,0]}}^{(1)})$.

If we replace one of $e_{1,1}$ by $e_{2,2}$, by taking its symmetrization, we have
$$E^*_{1,2} = E^*_{[1,0]} = L_{d-1,1}(e_{1,1},e_{2,2}),$$
which is a symmetrization of 
$$(E^*_{[0,0]})^{(d-1)}\otimes {E^*_{[1,2]}}^{(1)} = e_{1,1}\otimes e_{1,1}\otimes \cdots \otimes e_{1,1}\otimes e_{2,2}.$$
Similarly, 
$$E^*_{2,1} = E^*_{[0,1]} = L_{d-1,1}(e_{1,1},e_{3,3}),$$
which is a symmetrization of 
$$(E^*_{[0,0]})^{(d-1)}\otimes {E^*_{[2,1]}}^{(1)} = e_{1,1}\otimes e_{1,1}\otimes \cdots \otimes e_{1,1}\otimes e_{3,3}.$$
Now, we can proceed inductively to obtain the following by replacing one of $e_{1,1}$ by $e_{2,2}$, when one of the zeros is replaced by one, and one of $e_{1,1}$ by $e_{3,3}$, when one of the zeros is replaced by two. Therefore,
$$E^*_{[s,t]} = L_{r,s,t}({E^*_{[0,0]}}^{(1)},{E^*_{[1,2]}}^{(1)},{E^*_{[2,1]}}^{(1)}), $$
which is a symmetrization of 
$$e_{1,1}\otimes e_{1,1}\otimes \cdots \otimes e_{1,1}\otimes e_{2,2}\otimes \cdots \otimes e_{2,2}\otimes e_{3,3}\otimes \cdots \otimes e_{3,3}$$
with $r = d-s-t$ $e_{1,1}$'s, $s$ $e_{2,2}$'s and $t$ $e_{3,3}$'s, as desired.
\end{proof}

\begin{lemma}
For every $M\in \mat_3(\bC) = \cT(H^*(1,3))$, 
$$L(M) \in \cT(H^*(d,3)) \subseteq \mat_{3d}(\bC).$$
\end{lemma}
\begin{proof}
We first prove $L(M) \in \cT(H^*(d,3))$ for the generators $A^{(1)}$, ${A^{(1)}}^\top$, ${E^*_{[0,0]}}^{(1)}$, ${E^*_{[1,2]}}^{(1)}$, and ${E^*_{[2,1]}}^{(1)}$.

For $A^{(1)}$ and ${A^{(1)}}^\top$, the assertion follows from Lemm~\ref{lemma:a,eij}.

Since 
$${E^*_{[0,0]}}^{(1)} + {E^*_{[1,2]}}^{(1)} + {E^*_{[2,1]}}^{(1)} = e_{1,1} + e_{2,2} + e_{3,3} = I_3,$$
\begin{align*}
\text{LHS} & = I_3 \otimes e_{1,1} \otimes e_{1,1}\otimes \cdots \otimes e_{1,1}\otimes e_{2,2}\otimes \cdots \otimes e_{2,2}\otimes e_{3,3}\otimes \cdots \otimes e_{3,3}\\
& = e_{1,1} \otimes e_{1,1} \otimes e_{1,1}\otimes \cdots \otimes e_{1,1}\otimes e_{2,2}\otimes \cdots \otimes e_{2,2}\otimes e_{3,3}\otimes \cdots \otimes e_{3,3} \\
& \qquad + e_{2,2} \otimes e_{1,1} \otimes e_{1,1}\otimes \cdots \otimes e_{1,1}\otimes e_{2,2}\otimes \cdots \otimes e_{2,2}\otimes e_{3,3}\otimes \cdots \otimes e_{3,3}\\
& \qquad + e_{3,3} \otimes e_{1,1} \otimes e_{1,1}\otimes \cdots \otimes e_{1,1}\otimes e_{2,2}\otimes \cdots \otimes e_{2,2}\otimes e_{3,3}\otimes \cdots \otimes e_{3,3}.
\end{align*}
Thus, if $d\geq 2$, by setting $r = (d-1)-s-t$, the symmetrization of the LHS becomes
\begin{align*}
L_{r,s,t}(e_{1,1},e_{2,2},e_{3,3})  & \in  \text{Span}\{L_{r+1,s,t}(e_{1,1},e_{2,2},e_{3,3}), \; L_{r,s+1,t}(e_{1,1},e_{2,2},e_{3,3}), \\
& \qquad L_{r,s,t+1}(e_{1,1},e_{2,2},e_{3,3})\}\\
& \subseteq \cT(H^*(d,3)).
\end{align*}
Hence, $L_{r,s,t}(e_{1,1},e_{2,2},e_{3,3}) \in \cT(H^*(d,3))$ if $r+s+t = d-1$.

We can proceed inductively to show 
$$L_{r,s,t}(e_{1,1},e_{2,2},e_{3,3}) \in \cT(H^*(d,3)) \; \text{ if } r+s+t = 1.$$

Therefore, 
$$L({E^*_{[0,0]}}^{(1)}), \; L({E^*_{[1,2]}}^{(1)}), \; L({E^*_{[2,1]}}^{(1)})\in \cT(H^*(d,3)).$$ 

\medskip
Next, we show that for every matrix unit $e_{i,j}$ $(i,j\in \{1,2,3\})$ of $\mat_3(\bC) = \cT(H^*(1,3))$, $L(e_{i,j})\in \cT(H^*(d,3))$. We are done for $e_{1,1}$, $e_{2,2}$ and $e_{3,3}$.

Recall that, $e_{1,2} = e_{1,1}{A^\top}^{(1)}$, $e_{1,3} = e_{1,1}{A}^{(1)}$, $e_{2,1} = e_{2,2}{A}^{(1)}$, $ e_{2,3} = e_{2,2}{A^\top}^{(1)}$, $e_{3,1} = e_{3,3}{A^\top}^{(1)}$, $e_{3,2} = e_{3,3}{A}^{(1)}$. 

Let $e\in \{e_{1,1}, e_{2,2}, e_{3,3}\}$ and $M\in \{A^{(1)}, {A^\top}^{(1)}\}$. Note that $ee = e$ and $eMe = O$. We show $L(eM)\in \cT(H^*(d,3))$.

\begin{align*}
L(e)L(M) & = L(eM) + L_{1,1}(e,M),\\
L(M)L(e) & = L(Me) + L_{1,1}(e,M).
\end{align*}
Hence, by subtracting, we have $L(eM-Me) = L(eM)- L(Me) \in \cT(H^*(d,3))$.

\begin{align*}
L(eM-Me)L(e) & = -L(Me) + L_{1,1}(eM-Me,e),\\
L(e)L(eM-Me) & = L(eM) + L_{1,1}(eM-Me,e).
\end{align*}
By subtracting, we have $L(eM+Me) = L(eM)+L(Me)  \in \cT(H^*(d,3))$. 

Therefore, $L(eM), L(Me)\in \cT(H^*(d,3))$.

Since every matrix in $\mat_3(\bC)$ can be written as a linear combination of the matrix units, $L(M) \in \cT(H^*(d,3))$ for every $M \in \mat_3(\bC) = \cT(H^*(1,3))$.
\end{proof}

\medskip
The following is in \cite[Page 7]{lmp}. We include the same for completeness.

\begin{lemma}\label{lemma:zd}
For each $M\in \mat_3(\bC)$, we have
$$M^{\otimes d} = M\otimes M \otimes \cdots \otimes M \in \cT(H^*(d,3)).$$
\end{lemma}
\begin{proof}
By the previous lemma, we have that $L(M)\in \cT(H^*(d,3))$ for each $M\in \mat_3(\bC)$. 

By induction on $j$, we show $L_j(M)\in \cT(H^*(d,3))$ for every $M\in \mat_3(\bC)$. 
Assume it is true for $j\geq 1$. Then
\begin{align*}
L(M)L_j(M) & = L_{j-1,1}(M,M^2) + L_{j+1}(M)\\
L(M^2)L_{j-1}(M) & = L_{j-2,1}(M,M^3) + L_{j-1,1}(M,M^2)\\
L(M^3)L_{j-2}(M) & = L_{j-3,1}(M,M^4) + L_{j-2,1}(M,M^3)\\
 & \vdots\\
L(M^{j-1})L_2(M) & = L_{1,1}(M,M^j) + L_{2,1}(M,M^{j-1})\\
L(M^j)L(M) & = L(M^{j+1}) + L_{1,1}(M,M^j).
\end{align*}
Since $L(M^j)L(M)\in \cT(H^*(d,3))$ and $L(M^{j+1}) \in \cT(H^*(d,3))$, $L_{1,1}(M,M^j) \in \cT(H^*(d,3))$. Proceeding backwards, we have $L_{j+1}(M) \in \cT(H^*(d,3))$ as desired.
\end{proof}

\begin{prop}
The following hold.
$$\cT(H^*(d,3)) \simeq \mathrm{Sym}^{(d)}(\mat_3(\bC)).$$
\end{prop}
\begin{proof}
Since the generators of $\cT(H^*(d,3))$, i.e., $A, A^\top, E^*_{i,j}$ are in $\mathrm{Sym}^{(d)}(\mat_3(\bC))$ by Lemma~\ref{lemma:a,eij}.

The result is a direct consequence of Lemma~\ref{lemma:zd} 
%and Lemma 
in \cite{lmp}. 
\end{proof}

\begin{lemma}\label{lemma: formula}
$$\frac{1}{4}\sum_{\ell = 0}^{[d/3]}\sum_{m = 0}^{[(d-3\ell)/2]}(d-3\ell-2m+1)^2(m+1)^2(d-3\ell-m+2)^2 = \binom{d+8}{d}.$$
\end{lemma}
\begin{proof}
We prove this by induction. The formula is valid when $d = 0, 1, 2$. Suppose the formula is true for $d-3$. Then
\begin{align*}
\text{LHS} & = \frac{1}{4} \sum_{m=0}^{\left[\frac{d}{2}\right]} (d - 2 m + 1)^2 (m + 1)^2 (d - m + 2)^2  + \binom{d+5}{8}\\
& = \frac{(d + 2) (d + 4) (d^5 + 15 d^4 + 91 d^3 + 279 d^2 + 412 d + 210)}{1680} + \binom{d+5}{8}\\
& = \binom{d+8}{8}.
\end{align*}
We have the formula.
\end{proof}

\begin{lemma}\label{lemma:weight-vector}
Let $e_{i,j}$ $(i,j\in \{1,2,3\})$ be the matrix units of $\mat_3(\bC)$. Let 
\begin{align*}
N_1(s,t,u) & = se_{1,2} + ue_{1,3} + te_{2,3},\\
N_2(s,t,u) & = N_1(s,t,u) \otimes  I_3 + I_3 \otimes N_1(s,t,u), \\
N_3(s,t,u) & = N_1(s,t,u) \otimes I_9 + I_3 \otimes N_2(s,t,u) \\
& = N_1(s,t,u) \otimes I_3 \otimes I_3 + I_3 \otimes N_1(s,t,u) \otimes I_3 + I_3 \otimes  I_3 \otimes N_1(s,t,u), \\
D_1(s,t) & = s e_{1,1} + (-s+t) e_{2,2} -t e_{3,3},\\
D_2(s,t) & = D_1(s,t) \otimes I_3 + I_3 \otimes  D_1(s,t),\\
D_3(s,t) & = D_1(s,t) \otimes I_9 + I_3 \otimes D_2(s,t)\\
& = D_1(s,t) \otimes I_3 \otimes I_3 + I_3 \otimes D_1(s,t) \otimes I_3 + I_3 \otimes  I_3 \otimes D_1(s,t), \\
f & =  \sym^{(2)}(e_{1,2}\otimes e_{2,3} - e_{1,3}\otimes e_{2,2}) \text{ and} \\
g & = \sym^{(3)}(e_{1,1}\otimes e_{2,2}\otimes e_{3,3}) - \sym^{(3)}(e_{1,1}\otimes e_{2,3}\otimes e_{3,2}) \\
&- \sym^{(3)}(e_{1,2}\otimes e_{2,1}\otimes e_{3,3}) 
 + \sym^{(3)}(e_{1,2}\otimes e_{2,3}\otimes e_{3,1}) \\
&- \sym^{(3)}(e_{1,3}\otimes e_{2,2}\otimes e_{3,1}) + \sym^{(3)}(e_{1,3}\otimes e_{3,2}\otimes e_{2,1}).
\end{align*}
Then, the following hold.
\begin{itemize}
\item[{\rm (i)}] $N_1(s,t,u)e_{1,3} = e_{1,3}N_1(s',t',u') = O$ and $D_1(s,t)e_{1,3} = se_{1,3}$, $e_{1,3}D_1(-s',-t') = t'$.
\item[{\rm (ii)}] $N_2(s,t,u) f = f N_2(s',t',u') = O$ and $D_2(s,t)f = tf$, $fD_2(-s',-t') = s'f$.
\item[{\rm (iii)}] $N_3(s,t,u)g = gN_3(s',t',u') = O$ and $D_3(s,t)g = gD_3(-s',-t') = O$.
\end{itemize}
\end{lemma}
\begin{proof}
Straightforward.
\end{proof}

\bigskip\noindent
\begin{proof}[{\it Proof of Theorem~\ref{thm:uchida}.}]

Let $R = \mat_3(\bC)$. We first consider $\mathrm{Sym}^{(d)}(\mat_3(\bC))$ as an $R$-bi-module with respect to the left action; 
\begin{align*}
\lefteqn{x \cdot \text{Sym}^{(d)}(X_1, X_2, \cdots, X_d)}\\
 &= \text{Sym}^{(d)}(x\cdot X_1, X_2, \ldots, X_d) +  \text{Sym}^{(d)}(X_1, x\cdot X_2, \ldots, X_d) \\
&+ \cdots +  \text{Sym}^{(d)}(X_1, X_2, \ldots, x\cdot X_d),
\end{align*}
and the right action;
\begin{align*}
\lefteqn{\text{Sym}^{(d)}(X_1, X_2, \cdots, X_d)\cdot y}\\
 &= \text{Sym}^{(d)}(X_1\cdot y, X_2, \ldots, X_d) +  \text{Sym}^{(d)}(X_1, X_2\cdot y, \ldots, X_d) \\
&+ \cdots +  \text{Sym}^{(d)}(X_1, X_2, \ldots, X_d\cdot y)
\end{align*}
for $x,y \in R$.

For $\ell$ ($0\leq \ell \leq \left[\frac{d}{3}\right]$), and $m$ ($0\leq m\leq \left[\frac{d-3\ell}{2}\right]$), let $h$ be the symmetrization of the following matrix in $\mat_{3d}(\bC)$
$$e_{1,3}^{\otimes d-3\ell-2m}\otimes f^{\otimes m} \otimes g^{\otimes \ell}, $$
where $e_{1,3}\in \mat_3(\bC)$, $f\in \mat_9(\bC)$ and $g\in \mat_{27}(\bC)$ are matrices defined in Lemma \ref{lemma:weight-vector}.

Then, by Lemma~\ref{lemma:weight-vector}, it is a highest weight vector with weight $((d-3\ell-2m, m),(m,d-3\ell-2m))$. Thus, $Rh$ is an irreducible $R$ left module with highest weight $(d-3\ell-2m, m)$, and $hR$  is an irreducible $R$ right module with highest weight $(m,d-3\ell-2m)$. Since the dimension of the both modules is 
$$d_{\ell,m} = \frac{1}{2}(d-3\ell-2m + 1)(m+1)(d-3\ell-m+2)$$
by Lemma~\ref{lem:deg}, the dimension of $RhR$ is isomorphic to the matrix algebra of size $d_{\ell,m}$ with dimension $\frac{1}{4}(d-3\ell-2m+1)^2(m+1)^2(d-3\ell-m+2)^2$.

Since 
$$\frac{1}{4}\sum_{\ell = 0}^{[d/3]}\sum_{m = 0}^{[(d-3\ell)/2]}(d-3\ell-2m+1)^2(m+1)^2(d-3\ell-m+2)^2 = \binom{d+8}{d}$$
by Lemma~\ref{lemma: formula}, and the right hand side is $\dim \mathrm{Sym}^{(d)}(\mat_3(\bC))$, we have the desired isomorphism. 
\end{proof}

\begin{comment}

\bigskip\noindent
{\it The Second Proof. \quad}
By Lemma~\ref{lemma:existence}, Proposition~\ref{prop:basis}, and Lemma~\ref{lemma:parameters}, each irreducible module is parametrized by $m_1$ and $m_2$ and the dimension is $\frac12(m_1+1)(m_2+1)(m_1+m_2+2)$. Hence, the total dimension is the sum of the square of the dimension of the irreducible modules; we have the assertion.

Q.E.D.

\end{comment}

%\section{Problems}

Finally, we present some open problems:
\begin{prob}
\begin{enumerate}

\item Find the multiplicity formula of the standard module of $H^*(d,3)$.

\item Study the relationship between the Terwilliger algebra of a weakly distance-regular digraph $H^*(d,3)$ and the underlying distance-regular graph $H(d,3)$.
\item Study the structure of the irreducible modules of the Terwilliger algebra of a weakly distance-regular digraph of Hamming type determined in \cite{yqw}.

\item Determine the structure of the Terwilliger algebra of a weakly distance-regular digraph with respect to every vertex.
%\item Study thin weakly distance-regular digraphs.
\end{enumerate}
\end{prob}
For the recent progress of Problem (3) and (4), see \cite{mnssw2,mnssw3}. 

%%%%%%%%%%%%%%%%%%%%%%%%%%%%%%%%%

\section*{Acknowledgements}
%The authors thank Koji Chinen and Iwan Duursma for their helpful discussions 
%and contributions to this research.
\noindent
The first named author is supported by JSPS KAKENHI (22K03277).

%The authors would also like to thank the anonymous
%reviewers for their beneficial comments on an earlier version of the manuscript.

%\section*{Data availability statement}
%The data that supports the findings of this study are available from
%the corresponding author.

\end{document}